\documentclass[
submission
]{dmtcs-episciences}

\usepackage[utf8]{inputenc}
\usepackage{subfigure}

\usepackage{amsmath, amssymb,amsthm}

\usepackage{tikz}

\newtheorem{theorem}{Theorem}
\newtheorem{proposition}{Proposition}
\newtheorem{remark}{Remark}
\newtheorem{observation}{Observation}
\newtheorem{corollary}{Corollary}

\author{Sandip Das\affiliationmark{1}
  \and Siddani Bhaskara Rao\affiliationmark{2}
  \and Uma kant Sahoo\thanks{Email: umakant.iitkgp@gmail.com} \affiliationmark{1}}
\title[Pseudoline arrangement graphs: degree sequences and eccentricities]{Pseudoline arrangement graphs: degree sequences and eccentricities\thanks{An \href{https://drive.google.com/file/d/1QwGoaH6zfsgNmWVbvPf1dJP29jyLF-Pj/view?usp=sharing}{extended abstract}~\cite{das_degree_2021} was presented in CALDAM~2021, containing the proof of the corresponding graph realization problem.}}
\affiliation{
  Indian Statistical Institute, Kolkata, India\\
  CRRAO Advanced Institute of Mathematics, Statistics and Computer Science, Hyderabad, India
 }
\keywords{pseudoline arrangements, arrangement graphs, graph realization problem, eccentricity, diameter, radius}
\received{xxxx-xx-xx}
\revised{xxxx-xx-xx, 2014-02-05, 2015-09-09}
\accepted{xxxx-xx-xx}
\begin{document}
\publicationdetails{VOL}{2015}{ISS}{NUM}{SUBM}
\maketitle
\begin{abstract}
  A pseudoline arrangement graph is a planar graph induced by an embedding of a (simple) pseudoline arrangement. 
  We study the corresponding graph realization problem and properties of pseudoline arrangement graphs. 
  In the first part, we give a simple criterion based on the degree sequence that says whether a degree sequence will have a pseudoline arrangement graph as one of its realizations. 
  In the second part, we study the eccentricities of vertices in such graphs. 
  We observe that the diameter (maximum eccentricity of a vertex in the graph) of any pseudoline arrangement graph on $n$ pseudolines is $n-2$. 
  Then we characterize the diametrical vertices (whose eccentricity is equal to the graph diameter) of pseudoline arrangement graphs. 
  These results hold for line arrangement graphs as well. 
\end{abstract}

\section{Introduction}
\label{sec: 1 Introduction}
Both \textit{line} and \textit{pseudoline arrangements} are basic objects of study in discrete and computational geometry (see \cite[Chap. 5, 6]{felsner_geometric_2004}; also see \cite{felsner_pseudoline_2004} and the references therein). 
Their embeddings in the plane are a natural source of graphs. 
These graphs form a well-structured family of planar graphs having many interesting properties: starting from recognition~\cite{bose_properties_2003,eppstein_drawing_2014} to other graph characteristics~\cite{bose_properties_2003,felsner_hamiltonicity_2006}. 
They are also used to find the computational complexities of various geometric graph parameters~\cite{chaplick_complexity_2017, durocher_note_2013, okamoto_variants_2019}. 
However, these graphs are not that well studied. 
In this article, we study the corresponding \textit{graph realization problem} (described later) and the eccentricities of vertices in such graphs, which are formed by the embeddings of pseudoline arrangements in the Euclidean plane $\mathbb{R}^2$.

A \emph{pseudoline}\footnote{The definition used here is equivalent to the following: A pseudoline is the image of a line under a homeomorphism of the Euclidean plane. This is less restrictive~\cite{eppstein_drawing_2014} than the alternative definition: A pseudoline is a
non-contractible simple closed curve in the projective plane.} is a curve that approaches a point at infinity in either direction. 
An \emph{arrangement} $\mathcal{A}(L)$ of pseudolines in the Euclidean plane $\mathbb{R}^2$ is a collection $L$ of (at least three) pairwise intersecting pseudolines. Each pair of pseudolines intersect exactly once where they \emph{cross} each other. 
It is \emph{simple} if no three pseudolines meet at a point. 
The class of \emph{pseudoline arrangement graphs} $\mathcal{G}_{\mathcal{L}}$ are graphs induced by simple arrangements $\mathcal{A}(L)$, for any set of pseudolines $L$, whose vertices are intersection points of pseudolines in $L$, and there is an edge between two vertices if they appear on one of the pseudolines, say  $l\in L$, with no other vertices in the part of $l$ between the two vertices. 
The realization of a pseudoline arrangement graph $G_{L}$ by pseudolines in $L$ is its \emph{pseudoline arrangement realization} $R(G_{L})$. 
We get a pseudoline arrangement realization from its corresponding pseudoline arrangement by deleting the two infinite segments of each pseudoline. 
We have analogous definitions by replacing pseudolines with lines. 
Using a classical result of Whitney~\cite{whitney_congruent_1932} on planar graphs, Bose et al.~\cite{bose_properties_2003} showed that given a line arrangement graph, its line arrangement realization is unique up to isomorphism. 
Eppstein~\cite{eppstein_drawing_2014} extended this result to show the uniqueness of pseudoline arrangement realization (up to isomorphism).

\medskip
Pseudoline arrangements naturally generalize line arrangements and preserve their basic topological and combinatorial properties. 
It is well known that pseudoline arrangements strictly contain line arrangements (see~\cite{felsner_geometric_2004, grunbaum_arrangements_1972}). 
This relation is inherited by their corresponding graph classes.  
We focus on the general class: pseudoline arrangement graphs. 
The problems addressed in this article are graph-theoretic in nature, and their proofs have a geometric and topological flavor. 
As expected, our results hold for both the graph classes --- they do not depend on the \textit{straightness} of the lines. 
(In contrast, the computational complexities differ for their corresponding recognition problems; see Section~\ref{subsec: survey} for details.)

\medskip
\paragraph{Summary of our results.} We study the corresponding \emph{graph realization problem} on pseudoline arrangement graphs and the \emph{eccentricities} of its vertices. 
In particular, we prove that given a finite sequence of numbers, whether there is a pseudoline arrangement graph whose degrees correspond to the numbers in the sequence. 
We present this result in Section~\ref{ssec:degreeSequence} and prove it in Section~\ref{sec:degreeSequence}. 
We find the graph diameter of pseudoline arrangement graphs. 
Surprisingly, the diameter depends only on the number of pseudolines in its realization, and not on the graph structure. 
Our main result characterizes the diametrical vertices of a pseudoline arrangement graph, that is, the vertices whose eccentricity is equal to the graph diameter. 
We present these results on eccentricity in Section~\ref{ssec:eccentricity} and prove them in Section~\ref{sec:eccentricity}.

\paragraph{Organization.}  In the rest of this section, we summarize the relevant known results in the pseudoline arrangement graphs. In Section~\ref{sec:our results}, we state our results. In Section~\ref{sec:preliminaries}, we introduce some necessary tools and definitions. In Section~\ref{sec:degreeSequence}, we prove Theorem~\ref{th:degreeSequence}. And in Section~\ref{sec:eccentricity}, we present our results on the eccentricity of vertices in pseudoline arrangement graphs, leading to proofs of Proposition~\ref{th:diameter} and Theorem~\ref{th:outerlayer}. 
We conclude with some remarks in Section~\ref{sec:final}.

\medskip
\subsection{Survey of known results} \label{subsec: survey}
Steiner~\cite{steiner_einige_1826} studied line arrangements in 1826, and Levi~\cite{levi_teilung_1926} introduced pseudoline arrangements in 1926. 
However, it was the survey book by Gr\"unbaum~\cite{grunbaum_arrangements_1972} and,a few years later, the topological representation theorem of Folkman and Lawrence~\cite{folkman_oriented_1978} (amongst others) that have driven research in this field in the last five decades. 
For further details on line and pseudoline arrangements, see the surveys by Erd\H{o}s and Purdy~\cite{erdos_extremal_1995}, and Felsner and Goodman~\cite{felsner_pseudoline_2004}, and the book by Felsner~\cite[Chap. 5 and 6]{felsner_geometric_2004}. For a brief survey, see Section~\ref{ssec: survey pseudoline arrangements}. Before moving on to arrangement graphs, we want to highlight an important computational complexity question on arrangements, namely \emph{stretchability}: whether a pseudoline arrangement can be converted to an isomorphic line arrangement. Shor~\cite{shor_stretchability_1990} proved this problem to be NP-hard, and Schaefer~\cite{schaefer_complexity_2009} proved it to be $\exists \mathbb{R}$-hard. 

\smallskip
Now we discuss the known results on the focus of this paper: arrangement graphs. 
Bose et al.~\cite{bose_properties_2003} introduced the notion of a line arrangement graph in \textit{EuroCG}, 1998. 
This graph definition almost resembles the one given by Eu, Gr\'evremont and Toussaint~\cite{eu_envelopes_1996}, who gave an efficient algorithm for finding the envelope of a line arrangement, that is, the outer face of the line arrangement realization. 
Bose et al.~\cite{bose_properties_2003} proved that recognizing line arrangement graphs is NP-hard by reduction from \textit{simple stretchability}. 
Amongst other results, they gave examples of non-Hamiltonian line arrangement graphs.

Soon Felsner et al.~\cite{felsner_hamiltonicity_2006} showed pseudoline arrangement graphs are $4$-edge-colorable and 3-vertex-colorable. 
They also studied corresponding graphs got from other ambient spaces. 
They showed that projective pseudoline arrangement graphs are 4-connected and 4-vertex-colorable, and when the number of pseudolines is odd, they can decompose into two edge-disjoint Hamiltonian paths. 
They also showed that \textit{circle arrangement graphs} (great circles on a sphere) are 4-connected, 4-edge-colorable and 4-vertex-colorable, and their generalization,  \textit{pseudocircle arrangement graphs}, decompose into two edge-disjoint Hamiltonian cycles. 
Eppstein~\cite{eppstein_drawing_2014} gave a linear-time algorithm to draw a (pseudo) line arrangement graph in a grid of area $O(n^{7/6})$. 
He also proved that any pseudoline arrangement graph causing the algorithm to use $\Omega(n^{1+\epsilon})$ area would imply significant progress in the \emph{$k$-set problem} (see~\cite[Chap. 11]{matousek_lectures_2002}) in combinatorial geometry. 

Coming back to the computational complexity aspects of these graph classes, it follows from Schaefer~\cite{schaefer_complexity_2009} that \emph{simple stretchability} is $\exists\mathbb{R}$-hard. 
Hence the reduction of Bose et al.~\cite{bose_properties_2003} also implies that recognizing line arrangement graphs is $\exists\mathbb{R}$-hard (see \cite[p. 212]{eppstein_drawing_2014}). 
This result/reduction is used (i) by Durocher et al.~\cite{durocher_note_2013} to prove that checking whether there is a straight-line drawing of a planar graph with at most $k$ segments is NP-hard (in fact, it is $\exists\mathbb{R}$-hard), (ii) by Chaplick et al.~\cite{chaplick_complexity_2017} to prove that the \textit{line cover number} of a planar graph in two dimensions and a graph in three dimensions are $\exists\mathbb{R}$-hard to compute, and (iii) by Okamoto, Ravsky and Wolff~\cite{okamoto_variants_2019} to prove that many variants of the \textit{segment number} of a planar graph are $\exists\mathbb{R}$-hard to compute. 
On the other hand, Eppstein~\cite{eppstein_drawing_2014} proved that pseudoline arrangement graphs can be recognized in linear time. 
At its core, this recognition algorithm builds upon the ideas of Bose et al.~\cite{bose_properties_2003}.

\section{Our Results}\label{sec:our results}

\subsection{The Pseudoline Arrangement Graph Realization Problem}\label{ssec:degreeSequence}
The \emph{degree sequence} of a graph is the non-increasing list of degrees of its vertices. 
A graph with degree sequence $\pi$ is a \textit{realization} of $\pi$. 
Given an arbitrary finite sequence of non-increasing numbers $\pi$, the \emph{graph realization problem} asks whether a graph realizes $\pi$. 
Researchers have studied this classical problem from graph theory for the past six decades. 
The Erd\H{o}s-Gallai theorem~\cite{erdos_graphs_1961} and the Havel-Hakimi algorithm~\cite{havel_remark_1955, hakimi_realizability_1962} (strengthening of the former) are two popular methods to solve the graph realization problem. 
We discuss a similar problem. 
\paragraph{Pseudoline Arrangement Graph Realization Problem.} Given a sequence of finite numbers $\pi$, whether there is a pseudoline arrangement graph with degree sequence $\pi$. 

\smallskip
The following theorem solves this problem. 
For an affirmative answer, we construct a (pseudo) line arrangement realization. 
We give this construction and the proof in Section~\ref{sec:degreeSequence}. 

A vertex with degree $i$ is an \textit{$i$-vertex}, for $2\leq i\leq 4$; let $d_i$ denote the number of $i$-vertices. Let $\langle a^d \rangle$ denote the sequence $\langle a,\ldots,a \rangle$ of length $d$.

\begin{theorem}\label{th:degreeSequence}
	A finite non-increasing sequence of positive numbers $\pi$ is a degree sequence of a pseudoline arrangement graph if and only if it satisfies the following two conditions. 
	\begin{enumerate}
		\item $\pi = \left\langle 4^{d_4}, 3^{d_3}, 2^{d_2} \right\rangle$ with  $3\leq d_2 \leq n$, $d_3 = 2(n-d_2)$ and $d_4 = n(n-5)/2 + d_2$ for some integer $n \geq 3$.
		\item If $d_2=n$, then $n$ is odd.
	\end{enumerate}
\end{theorem}

As expected Theorem~\ref{th:degreeSequence} also holds for line arrangement graphs (see Remark~\ref{re:linedegree}) and so does the following discussions. 

\smallskip
Other graph classes have stronger characterizations based on the degree sequences. 
A graph class $\mathcal{G}$ has a \emph{degree sequence characterization} if one can recognize whether a graph $G \in \mathcal{G}$ or not, based on its degree sequence. 
Hence to recognize whether $G \in \mathcal{G}$ or not, one needs to check if the degree sequence of $G$ satisfies all the conditions of the degree sequence characterizations. This often leads to linear-time recognition algorithms~\cite{bose_characterization_2008,hammer_splittance_1981,merris_split_2003}. 
(We present a brief review in Section~\ref{subsec:degreeSequence}.) 
However, for the following reason, we cannot infer anything about the recognition of pseudoline arrangement graphs from Theorem~\ref{th:degreeSequence}.

A \emph{$2$--switch} operation replaces a pair of edges $xy$ and $zw$ in a simple graph by the edges $yz$ and $wx$, given that $yz$ and $wx$ were not edges in the graph. 
Performing a $2$--switch operation in a graph does not change its degree sequence.  
The class of pseudoline arrangement graphs $\mathcal{G}_{\mathcal{L}}$ is not \emph{closed} under the \emph{$2$--switch} operation, that is, after performing a \emph{$2$--switch} operation in $G_{L}\in \mathcal{G}_{\mathcal{L}}$, the resulting graph may not be in $\mathcal{G}_{\mathcal{L}}$ (easy to check on the arrangement graph induced on four pseudolines). 
This kills all the hope for obtaining a degree sequence characterization for pseudoline arrangement graphs.  
Thus, in this ``sense", Theorem~\ref{th:degreeSequence} is the best one can hope for. 
This is also strongly indicated by the following: 
Theorem~\ref{th:degreeSequence} also holds for line arrangement graphs (see Remark~\ref{re:linedegree}), which are $\exists \mathbb{R}$-hard (and hence NP-hard) to recognize (see~\cite{bose_properties_2003,schaefer_complexity_2009}).  

\smallskip

We further want to highlight that Theorem~\ref{th:degreeSequence} also implies that the class of pseudoline arrangement graphs cannot have a \emph{forbidden graph characterization}, that is, a characterization for recognizing a graph class by specifying a list of graphs that are forbidden to exist as (or precisely, be isomorphic to) an induced subgraph of any graph in the class. A result of Greenwell et al.~\cite{greenwell_forbidden_1973} says that a graph class has a forbidden graph characterization if and only if it is closed under taking induced subgraphs. 
The pseudoline arrangement graphs are not closed under vertex deletions. 
Indeed, Theorem~\ref{th:degreeSequence} implies that deleting any vertex in a pseudoline arrangement graph does not result in a pseudoline arrangement graph. 
Hence pseudoline arrangement graphs cannot have a forbidden graph characterization.

\medskip
The graph realization problem is just a preliminary query in the more general framework of degree-based graph construction problem in \emph{network modeling}~\cite{kim_degree-based_2009}. Given a sequence $\pi$, let $\mathcal{N}$ be the set of realizations of $\pi$ (up to isomorphism) that satisfy some conditions. A degree sequence-based graph construction problem asks (i) if $\mathcal{N}\neq \emptyset$, (ii) if it is possible to construct a member of $\mathcal{N}$, (iii) to find (asymptotics of) $|\mathcal{N}|$, (iv) if there is a way to construct all (or a fraction) of graphs in $\mathcal{N}$, and other questions. We have addressed the first two questions for (pseudo) line arrangement graphs. We leave the other two questions as open problems. The separating examples of line arrangement graphs and pseudoline arrangement graphs imply that the answers to (iii) and perhaps (iv) are going to be different for the two graph classes.

\subsection{Eccentricities in Pseudoline Arrangement Graphs}\label{ssec:eccentricity}

The \emph{distance} $d(u,v)$ between two vertices $u,v$ of a graph $G$ is the length of the shortest path between them. 
The \emph{eccentricity} $e(u)$ of a vertex $u$ is the maximum distance of a vertex in $V(G)$ from $u$. 
A vertex $v$ is an \emph{eccentric vertex} of $u$ if $d(u,v)=e(u)$. 
The \emph{diameter} $d(G)$ of $G$ is the maximum eccentricity of any vertex in $V(G)$. 
A vertex $u$ is \emph{diametrical} if $e(u)=d(G)$. 
The \emph{radius} $r(G)$ of $G$ is the minimum eccentricity of any vertex in $V(G)$. 
A vertex $u$ is \emph{central} if $e(u)=r(G)$. 

Recall that Bose et al.~\cite{bose_properties_2003} were the first to introduce line arrangement graphs and its definition resembles the one given by Eu, Gr\'evremont and Toussaint~\cite{eu_envelopes_1996} who gave an efficient algorithm for finding the envelope of a line arrangement. 
This problem was earlier studied by Ching and Lee~\cite{ching_finding_1985} in 1985. 
However, their focus was on finding the Euclidean diameter of a line arrangement. 
We begin our study of eccentricities in pseudoline arrangement graphs by studying the graph-theoretic analog of this classic computational geometry problem.

\begin{proposition}    \label{th:diameter}%
	The diameter of a pseudoline arrangement graph on $n$ pseudolines is $n-2$. 
\end{proposition}

Surprisingly, the diameter of a pseudoline arrangement graph is independent of the graph and depends only on the number of pseudolines in its realization. 

As a prelude to Proposition~\ref{th:diameter}, we begin with some basic observations regarding the properties of shortest paths and eccentric vertices in Section~\ref{subsec:basic results on shortest paths}. They vary from the restrictions on the shortest paths between two vertices to the existence of particular types of eccentric vertices. These observations are also of independent interest. Using these observations, or otherwise, we prove Proposition~\ref{th:diameter} in Section~\ref{subsec:diameter and radius}.

\medskip

Our next aim is to find the radius of pseudoline arrangement graphs. 
Unlike the diameter, one can see that the radius of a pseudoline arrangement graph will depend on the graph structure. 
Our central idea is to prove that as we move to the interior of the pseudoline arrangement realization, after iteratively removing the outer layer of vertices, one would expect the eccentricity of vertices in the inner layers to decrease. 
As a first step, we begin by characterizing diametrical vertices (whose eccentricity is $n-2$) in the pseudoline arrangement graphs. 
We prove it in Section~\ref{proof:outerlayer}. 

\begin{theorem}\label{th:outerlayer}
	A vertex $v$ in a pseudoline arrangement graph $G$ is a diametrical vertex if and only if $v$ lies in the outer face of its realization $R(G)$. 
\end{theorem}

Observe that Theorem~\ref{th:outerlayer} fixes the vertices that occur in the outer face of the realization of a (pseudo) line arrangement graph. 
In other words, it characterizes the intersection points of (pseudo) line arrangements that lie in the \textit{envelope} of the arrangement. 
Coincidentally, finding the envelope of a line arrangement is a subproblem pursued by Ching and Lee~\cite{ching_finding_1985} while finding the Euclidean diameter of line arrangements.    

\smallskip

However, we are still to prove any non-trivial bounds on the radius. We hope Theorem~\ref{th:outerlayer} to be a starting point for such a result. 
We suspect the pseudoline arrangement graph got by the \textit{star construction} (defined later) to have the maximum radius. We leave this as an open problem. 
Another open problem is to characterize the central vertices in a pseudoline arrangement graph, that is, vertices whose eccentricity equals the graph radius.  

\medskip
\paragraph{Future works on generalization.} 
We can pose similar questions on the generalized non-simple (pseudo) line arrangements. 
In particular, the question of the graph realization problem and diameter are interesting for general (pseudo) line arrangements. 
For the graph realization problem of general line arrangements, a natural hurdle in fixing the necessary conditions seems to be the  Dirac-Motzkin conjecture~\cite{dirac_collinearity_1951} on \textit{ordinary lines} (for all $n$). 
However, the graph realization problem of general pseudoline arrangements does not have this issue. 

The diameter of general (pseudo) line arrangement graphs will depend on the graph structure (unlike their simple counterparts). 
Hence finding tight lower bounds on the diameter for these (non-trivial) cases seems interesting.  We keep these lines of questioning for future work.

\section{Preliminaries}\label{sec:preliminaries}
\subsection{Tools Used}
We need the following common notions and constructions for pseudoline arrangements. 
See the book by Felsner~\cite[Chap. 6]{felsner_geometric_2004} for definitions and detailed constructions. 
Here we give a succinct description. 
For a fixed unbounded region $f$ of a pseudoline arrangement $\mathcal{A}$ on $n$ pseudolines, there is always an unbounded region $f^*$ that is \emph{separated} from $f$ by all pseudolines. 
Note that the boundaries of $f$ and $f^*$ have two pseudolines in common. 
Fix points $x\in f$ and $x^*\in f^*$. 
We \textit{topologically sweep} the arrangement to form an aesthetic arrangement of polylines (pseudolines made up of line segments) called the \textit{wiring diagram}~\cite{goodman_proof_1980} corresponding to the sweep. 
This process uses \emph{allowable sequences}~\cite{goodman_semispaces_1984}, which we do not describe here. 

Consider the internally disjoint oriented $x^*,x$-curves that do not contain any vertex of the arrangement and that crosses each pseudoline once. 
A \textit{topological sweep} of the arrangement is a sequence $c_0, \ldots, c_r$ of such oriented $x^*,x$-curves with $r=\binom{n}{2}$ such that there is one vertex between the curves $c_i$ and $c_{i+1}$. 
Here $c_0$ is the oriented $x^*,x$-curve such that all the vertices in $\mathcal{A}$ lie to the right of $c_0$ (with respect to the orientation of $c_0$). 
Label the pseudolines from $1$ to $n$ such that $c_0$ intersects the pseudolines in increasing order. 
Next, we form the wiring diagram corresponding to this topological sweep. 

Fix $n$ horizontal wires. 
We confine the pseudolines to these wires, except for the parts where they cross each other. 
Corresponding to the topological sweep, we have a sequence $p_0, \ldots, p_r$ of vertical lines with $p_i$ to the left of $p_j$, for $i<j$.  
The ordering of polylines in which $p_0$ intersects from bottom to top is $1, 2, \ldots, n$. 
Between $p_i$ and $p_{i+1}$, we allow only the two pseudolines that form the vertex between $c_i$ and $c_{i+1}$ to intersect. 
Hence the ordering of the polylines that intersects $p_i$ from bottom to top is the same as the ordering of pseudolines that intersects $c_i$ from $x^*$ to $x$. 
We call this the \textit{wiring diagram} corresponding to the topological sweep. See Figure~\ref{fig0} for an illustration.

\begin{figure}
	\centering
	
	\includegraphics[scale=0.7]{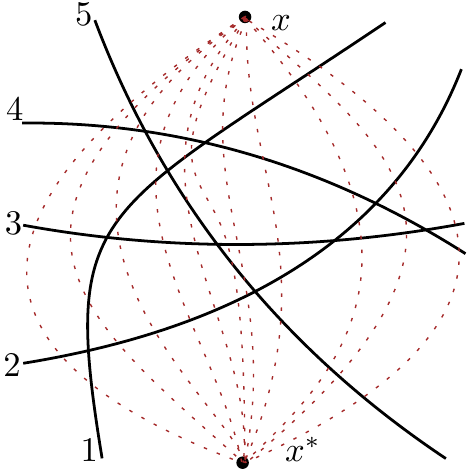}\hspace{1cm} \includegraphics[scale=0.7]{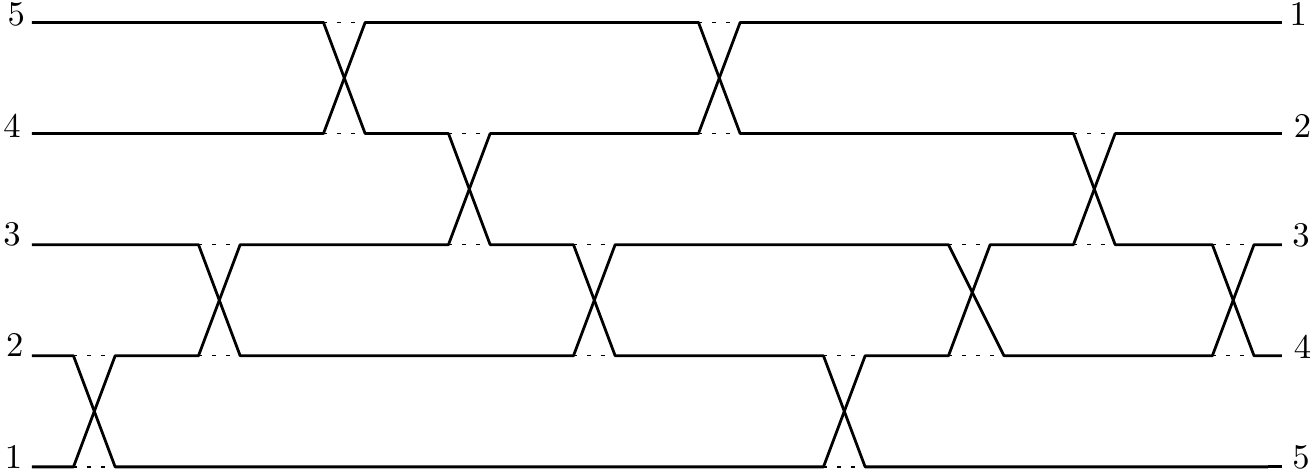}
	\caption{Topological sweep of a pseudoline arrangement and its wiring diagram.}
	\label{fig0}
	
\end{figure}

\subsection{Definitions and Notations}
Since we are dealing with pseudolines, which are topological analogs of lines, we shall come across terms like \emph{pseudohalfplane}, \emph{pseudoquadrant}, \emph{pseudotriangle}, \emph{pseudopolygon} etc., in our arguments; the prefix \emph{pseudo} denotes the topological analog of the following term. 

\smallskip

Let $G_{L}$ be a pseudoline arrangement graph with realization $R(G_{L})$. 
The \emph{span} of each pseudoline is the part of the pseudoline drawn in $R(G_{L})$, that is, the part of the pseudoline between its end vertices. 
A \emph{path} $P$ is a sequence of distinct vertices, such that consecutive vertices are adjacent. 
The \emph{length} of a path $P$, denoted $|P|$, is the number of edges in $P$. 
The length of the shortest $u,v$-path is $d(u,v)$. For vertices $u$ and $v$ in the path $P$ in $G_{L}$, let $uv|_P$ denote the $u,v$-path in $P$, with length $d(u,v)_P$. 
For vertices $u$ and $v$ in line $l \in L$, let $uv|_l$ denote the $u,v$-path in $l$, with length $d(u,v)_l$.  
For points $c$ and $d$ (may not be in $V(G_{L})$) that are on different pseudolines, let $\overline{cd}$ denote the line segment between $c$ and $d$. 
If a vertex $u \in V(G_{L})$ is an intersection point of two pseudolines $l_1$ and $l_2$, then we say $u = l_1\cap l_2$. 
For vertices $u, v \in V(G_{L})$ not lying on a pseudoline $l$, we say $l$ \emph{separates} $u$ and $v$ if they lie on different pseudohalfplanes bounded by $l$.

\smallskip

We give the relevant definitions specific to the proofs in their respective sections.

\section{Pseudoline Arrangement Graph Realization Problem}\label{sec:degreeSequence}

Now we solve the pseudoline arrangement graph realization problem, that is, we prove Theorem~\ref{th:degreeSequence}. 
For the sake of the reader, we restate Theorem~\ref{th:degreeSequence}.

\medskip	
\noindent \textbf{Theorem 1.}
		\textit{A finite non-increasing sequence of positive numbers $\pi$ is a degree sequence of a pseudoline arrangement graph if and only if it satisfies the following two conditions.} 
		\begin{enumerate}
			\item $\pi = \left\langle 4^{d_4}, 3^{d_3}, 2^{d_2} \right\rangle$ \textit{with}  $3\leq d_2 \leq n$, $d_3 = 2(n-d_2)$ \textit{and} $d_4 = n(n-5)/2 + d_2$ \textit{for some integer} $n \geq 3$.
			\item \textit{If} $d_2=n$, \textit{then} $n$ \textit{is odd}.
		\end{enumerate}

\begin{remark}\label{re:linedegree}
	Theorem~\ref{th:degreeSequence} also solves the line arrangement graph realization problem. However, a proof of Theorem~\ref{th:degreeSequence} for pseudoline arrangement graphs does not directly translate to a proof for line arrangement graphs. Recall that pseudoline arrangement graphs strictly contain line arrangement graphs. Hence to prove Theorem~\ref{th:degreeSequence} for both pseudoline and line arrangement graphs, we give a proof of necessity for pseudoline arrangement graphs and a proof of sufficiency for line arrangement graphs. 
\end{remark}

We devote the rest of this section to the proof of Theorem~\ref{th:degreeSequence}. 
In Section~\ref{ssec:sufficiency-degree}, we prove the necessity of Theorem~\ref{th:degreeSequence}. 
In Section~\ref{ssec:constructions}, we define some constructions and operations that we need to prove the sufficiency of Theorem~\ref{th:degreeSequence}, which we prove in Section~\ref{ssec:necessity-degree}. 
	
\subsection{Proof of Necessity}\label{ssec:sufficiency-degree}
	
Suppose the degree sequence $\pi$ has a pseudoline arrangement realization on $n\geq 3$ pseudolines. 
Since every pair of pseudolines intersect, $d_2 + d_3 + d_4 = n(n-1)/2$. 
Each of the end vertices of every pseudoline is either a $2$--vertex or a $3$--vertex: each $2$--vertex is an end vertex of two pseudolines, and each $3$--vertex is an end vertex of one pseudoline. 
Thus $2d_2 + d_3 = 2n$. 
From both these equations, $d_3 = 2(n-d_2)$, $d_4 = n(n-5)/2 + d_2$ and $d_2\leq n$. 

We claim $d_2\geq 3$. 
To see this, we first extend the realization to an arrangement on $n$ pseudolines. 
Next, extend the arrangement to a pseudoline arrangement on $n+1$ pseudolines in the real projective plane by adding an imaginary pseudoline $l$ at infinity. 
A standard result by Levi~\cite{levi_teilung_1926} (also see \cite[Prop. 5.13]{felsner_geometric_2004}) shows that every such pseudoline is incident to at least three triangles. 
In particular, $l$ is incident to at least three triangles in this arrangement in the real projective plane. 
Each of these triangles corresponds to unbounded regions with two pseudolines in its boundary in the arrangement in the Euclidean plane. 
The intersection point of these two pseudolines corresponds to a $2$-vertex in the realization. Hence $d_2\geq 3$. This is tight which can be verified by the following line arrangement on $n$ lines: tangents to the $n$ points uniformly distributed on a circular arc with a right angle. 
	
\smallskip
If $d_2 = n$ then $d_3 = 0$, that is, each of the end vertex of every pseudoline is a $2$--vertex. 
In such an arrangement consider any pseudoline $l$. 
Let $u$ and $v$ be the end vertices of $l$.  
Pseudoline $l$ divides the plane into two open pseudohalfplanes, denoted $l^+$ and $l^-$. 
The other two pseudolines from $u$ and $v$ meet in one of the pseudohalfplanes, say (without loss of generality) $l^+$. 
Let the number of $2$--vertices in $l^+$ be $k$. 
Thus the number of $2$--vertices in $l^-$ is $n-k-2$. 
All other pseudolines, except the three incident at $u$ or $v$, cross the part of $l$ between $u$ and $v$. 
We double count such crossings: for all the $2$-vertices in $l^-$ there are $2(n-k-2)$ such crossings, and for all the $2$--vertices in $l^+$ there are $2k-2$ crossings. 
Thus $2k-2 = 2(n-k-2)$. 
This implies $n = 2k + 1$. 
Hence if $d_2 = n$ then $n$ is odd.
This completes the necessity part of the proof. 

\begin{remark}
One can also derive the first condition of Theorem~\ref{th:degreeSequence} from the argument involving the projective plane in proving $d_2 \geq 3$. 
However, we want to highlight the approach using the two equations, as they hold for graphs induced by a more general arrangement of simple finite curves, which is helpful in alternately proving a result of Kostochka and Ne\v{s}et\v{r}il~\cite{kostochka_coloring_1998}. 
An alternate proof of $d_2\geq 3$ using \emph{wiring diagrams} is given in Section~\ref{alternate d2>3 pseudoline}. 
\end{remark}

\begin{remark}\label{rem:no projective}
Most of the first condition is intuitive and can also be concluded from the observations of Bose et al.~\cite{bose_properties_2003} and Durocher et al.~\cite{durocher_note_2013}. 
This highlights the importance of the second condition in Theorem~\ref{th:degreeSequence} for completing the characterization. 
\end{remark}

To prove the sufficiency of Theorem~\ref{th:degreeSequence}, we construct a line arrangement realization having the given degree sequence. 
First, we need the following construction and operations. Our aim is to first fix the $d_2$ degree two vertices. If $d_2$ is odd, we exploit the idea in the second condition of Theorem~\ref{th:degreeSequence} to build a ``star consruction". If $d_2$ is even, then we begin with a star construction on $d_2+1$ vertices and then remove one degree two vertex by using a ``pull operation". Once $d_2$ is fixed, we fix the $d_3$ degree three vertices by using the ``line operation". In such a operation, the number of degree two vertices remains intact while the number of degree three vertices increases following the equations in the first condition in Theorem~\ref{th:degreeSequence}. This automatically fixes the degree four vertices. 
	
\subsection{Construction and Operations}\label{ssec:constructions}
	
\paragraph{Star construction.} 
For odd $n$, place $n$ vertices uniformly on a circle and join each vertex to its opposite two farthest vertices. 
This results in a line arrangement realization on $n$ lines called a \emph{star construction on $n$ vertices}. 
The center of the circle is called the \emph{center of the star construction}. 
The rest of the vertices are $4$--vertices that lie within the circle. 
The star construction on $n$ vertices has degree sequence $\left\langle 4^{n(n-3)/2},2^{n}\right\rangle $. 
\begin{figure}
	\centering
	\begin{minipage}[b]{0.45\textwidth}
		\includegraphics[width=\textwidth]{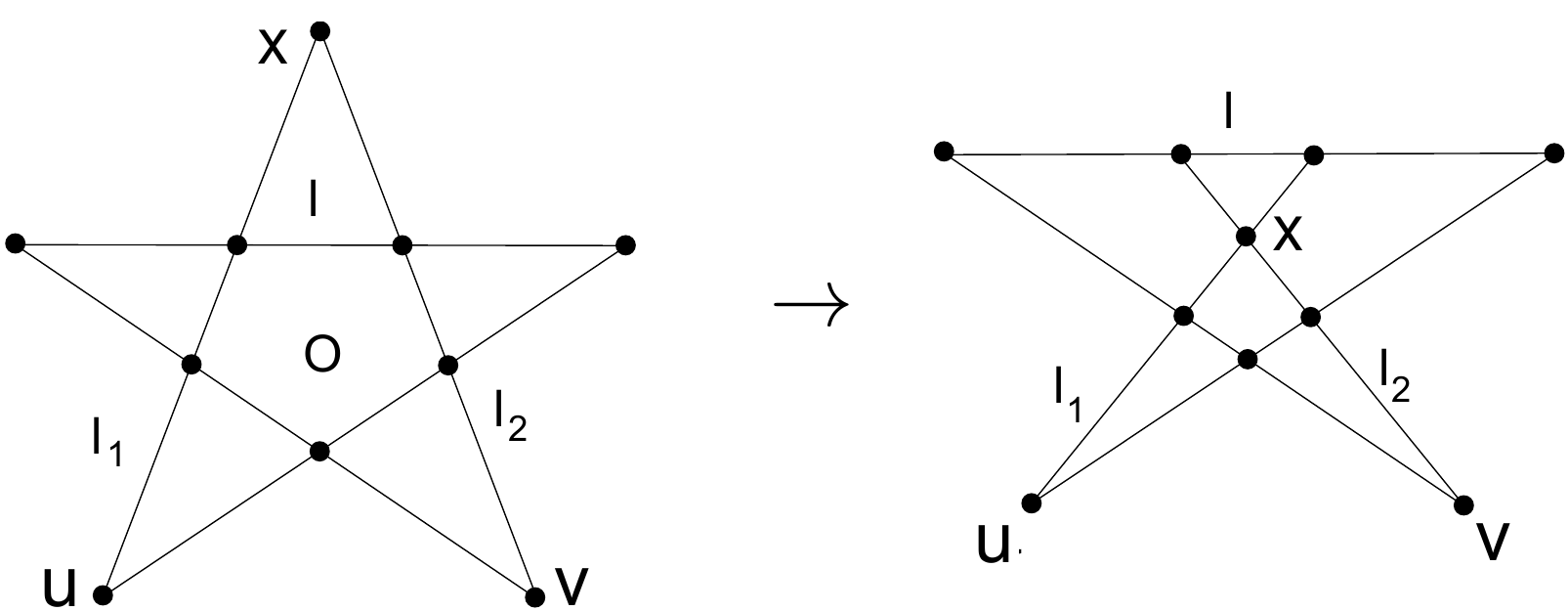}
		\caption{Pull Operation}
		\label{fig1}
		\end{minipage}
		\hfill
		\begin{minipage}[b]{0.5\textwidth}
			\includegraphics[width=\textwidth]{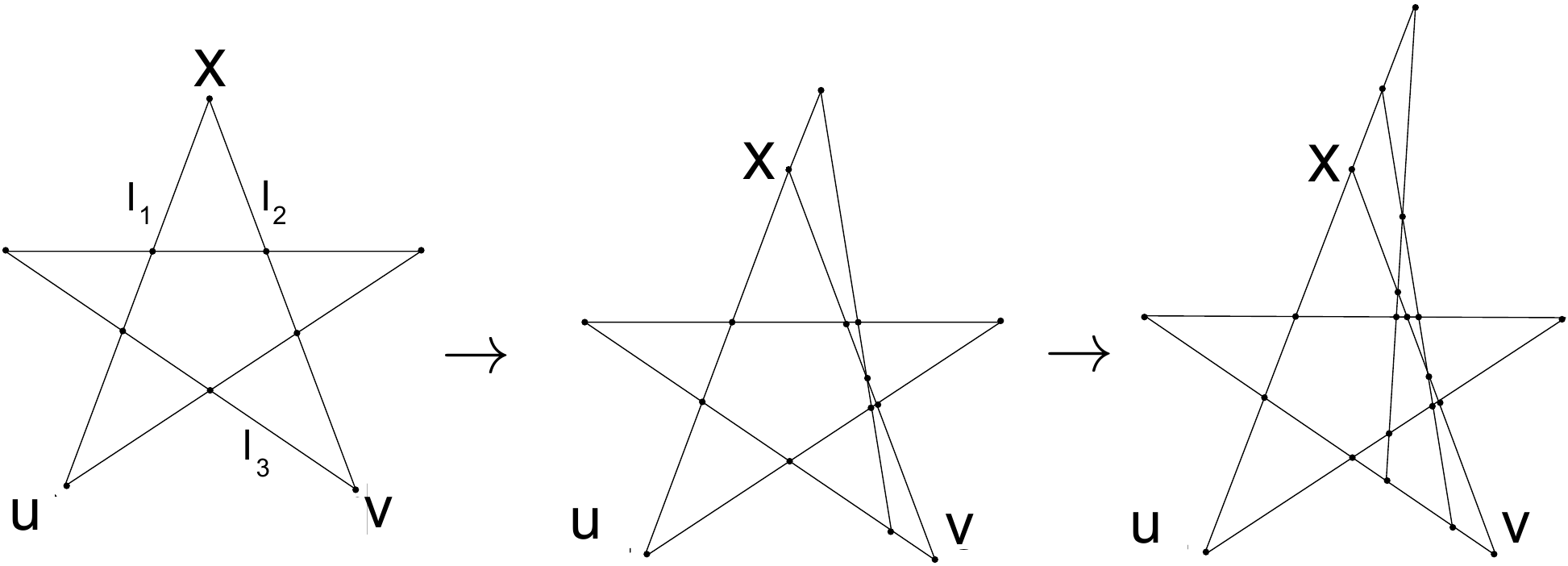}
			\caption{Two line operations on the $2$-vertex $x$}
			\label{fig2}
		\end{minipage}
	\end{figure}

\paragraph{Pull Operation.} 
Consider a $2$--vertex $x$ in a star construction on at least $5$ vertices with $O$ as the center of the star construction. 
Let $x= l_1\cap l_2$ with the other end vertices of $l_1$ and $l_2$ as $u$ and $v$, respectively.  
Let $l$ be the first line crossed while moving from $x$ to $O$ along the line segment $\overline{xO}$. 
Rotate $l_1$ and $l_2$ about $u$ and $v$ such that $x = l_1\cap l_2$ comes closer to $O$ till $x$ crosses $l$, while keeping the slope of $\overline{xO}$ fixed (see Figure \ref{fig1}). 
Now $x$ becomes a $4$--vertex and two new $3$--vertices are created at the expense of two $4$--vertices in the star construction. 
In this operation, the number of $2$--vertices decreases by one. 
Hence the degree sequence changes from $\left\langle 4^{d_4},2^{d_2}\right\rangle $ to $\left\langle 4^{d_4-1},3^2,2^{d_2-1}\right\rangle $.

\paragraph{Line Operation.} 
In the target degree sequence, if $d_2$ is odd, then consider a star construction. 
Choose a $2$--vertex $x= l_1\cap l_2$ with the other end vertices of $l_1$ and $l_2$ as $u$ and $v$ respectively, which are also $2$-vertices. 
Let $l_3$ be the other line intersecting $l_2$ at $v$.   
Take a point on the line $l_1$ that is close to $x$ and just outside the span of $l_1$ in the realization. 
Also, take a point in the realization on the line $l_3$ and close to $v$. 
Joining these two points, we add a new line $l$ to the realization that intersects all the span of other lines except the span of $l_1$ in the realization of the star construction (as shown in Figure \ref{fig2}). 
In the new realization, $l_1\cap l$ forms a new $2$--vertex, making $x$ a $3$--vertex. 
Thus the number of $2$--vertices is unaffected. 
The other end vertex of $l$ is also a $3$--vertex, increasing the number of $3$--vertices by $2$. 
The rest of the new vertices introduced are $4$--vertices. 

By doing $k$ line operations on $x$, we add $k$ such new lines to the realization close to $x$. 
On constructing such new lines we make sure that their intersection points with $l_3$ (and with $l_1$), in order of their addition, form a monotonic sequence of points in $l_3$ (and in $l_1$). 
It ensures that the new line added also intersects the previously added lines before reaching its end vertex, which is a $3$--vertex (refer Figure \ref{fig2}). 
Upon performing $k$ line operations on a star construction on $d_2$ vertices, the number of $2$--vertices remains unchanged; the number of $3$--vertices increases by $2k$; and the number of $4$--vertices increases by $\binom{d_2+k}{2}-\binom{d_2}{2}-2k=k(d_2+\frac{k-5}{2})$. 
Hence the degree sequence changes from $\langle 4^{d_4},2^{d_2}\rangle $ to $\langle 4^{d_4+k(d_2+\frac{k-5}{2})},3^{2k},2^{d_2}\rangle $. 

\medskip
In the target sequence, if $d_2$ is even, then we first consider a star construction on $d_2+1$ with one pull operation on it.  
This realization has $d_2$ $2$--vertices. 
We can also perform line operations on this realization. 
Repeating the above calculations, upon performing $k$ line operations on a star construction on $d_2+1$ vertices with a pull operation, the number of $2$--vertices remains unchanged; the number of $3$--vertices increases by $2k$; and the number of $4$--vertices increases by $\binom{d_2+1+k}{2}-\binom{d_2+1}{2}-2k=k(d_2+\frac{k-3}{2})$.  
Hence the degree sequence changes from $\langle 4^{d_4},3^2,2^{d_2}\rangle $ to $\langle 4^{d_4+k(d_2+\frac{k-3}{2})},3^{2+2k},2^{d_2}\rangle $.

\smallskip
Now we are ready to prove the sufficiency part of the Theorem\ref{th:degreeSequence}. 
	
\subsection{Proof of Sufficiency}\label{ssec:necessity-degree}
Let $\pi$ be a degree sequence satisfying the properties given in Theorem~\ref{th:degreeSequence} for some value of $n$. 
We give an algorithm to draw a line arrangement realization with degree sequence $\pi$.

\paragraph{Algorithm.} 
For odd $d_2$, do a star construction on $d_2$ vertices.
If $d_2 = n$, then $n$ is odd, and we have the required line arrangement realization; else do $n-d_2$ line operations on a $2$--vertex of the star to get the required realization. 
For even $d_2$, do a star construction on $d_2+1$ vertices and then do a pull operation on one of the $2$--vertices, resulting in $d_2$ $2$--vertices. 
If $d_2 = n-1$, then we have the required line arrangement realization; else if $n > d_2+1$, then do $n-d_2-1$ line operations on a $2$--vertex of the star construction to get the required realization. 
This results in a line arrangement realization with degree sequence $\pi$.

\paragraph{Correctness.}
\sloppy For odd $d_2$, a star construction on $d_2$ vertices results in  the degree sequence $\pi_s=\langle 4^{d_2(d_2-3)/2},2^{d_2}\rangle$. 
If $d_2 = n$, then $\pi_s=\pi$. 
If $d_2 < n$, then performing $n-d_2$ line operations increases $d_4$ by $(n-d_2)( d_2+\frac{n-d_2-5}{2} ) = \frac{n^2}{2}-  \frac{d_2^2}{2} -\frac{5n}{2}+\frac{5d_2}{2}$; increases $d_3$ by $2(n-d_2)$; and $d_2$ remains same. 
This results in the degree sequence $\langle 4^{n(n-5)/2+d_2},3^{2(n-d_2)},2^{d_2}\rangle$. 
	
For even $d_2$, a pull operation on a star construction on $d_2+1\geq 5$ vertices results in degree sequence $\pi_{sp}=\langle 4^{(d_2+1)(d_2-2)/2-1},3^2,2^{d_2}\rangle$. 
If $d_2 = n-1$, then $\pi=\langle 4^{n(n-3)/2-1},3^{2},2^{n-1}\rangle=\langle 4^{n(n-5)/2},3^{2},2^{n-1}\rangle$, that is, $\pi_{sp}=\pi$.  
If $d_2 < n-1$, then performing $n-d_2-1$ line operations increases $d_4$ by $(n-d_2-1)( d_2+\frac{n-d_2-4}{2} ) = \frac{n^2}{2}-  \frac{d_2^2}{2} -\frac{5n}{2}+\frac{3d_2}{2}+2$; increases $d_3$ by $2(n-d_2-1)$; and $d_2$ remains same. 
This results in the degree sequence $\langle 4^{n(n-5)/2+d_2},3^{2(n-d_2)},2^{d_2}\rangle$. 
	
So the algorithm gives a realization with degree sequence $\pi$. 
This completes the proof of the sufficiency of Theorem \ref{th:degreeSequence} and hence the proof of Theorem~\ref{th:degreeSequence}. \hfill $\square$

\section{Eccentricities in Pseudoline Arrangement Graphs}\label{sec:eccentricity}
In this section we study eccentricities of vertices in pseudoline arrangement graphs. 
To find the eccentricity of a vertex, we shall find one of its eccentric vertices. 
For this purpose, we derive some basic results on the shortest paths and eccentric vertices in pseudoline arrangement graphs, which are of independent interest. 
First, we need the following definitions. 
	
\smallskip
	
Consider two vertices $u$ and $v$ in a pseudoline arrangement that do not lie on the same pseudoline. 
Let $u=l_1\cap l_2$ and $v=l_3\cap l_4$. 
For each vertex in the pseudoline arrangement, the two intersecting pseudolines divide the Euclidean plane into four pseudoquadrants. 
Let $Q_v$ denote the pseudoquadrant defined by $l_1$ and $l_2$ that contains vertex $v$. 
Similarly, let $Q_u$ denote the pseudoquadrant defined by $l_3$ and $l_4$ that contains vertex $u$. Let $Q_{uv}=Q_u\cap Q_v$. 
	
\subsection{Basic Results on Eccentricities}\label{subsec:basic results on shortest paths}    
\medskip
The following simple observation can be proved by strong induction. 
To maintain the flow, we defer its proof to the Appendix (see Section~\ref{proof:prop:shortest}). 
\begin{proposition}\label{prop:shortest}
	For vertices $u$ and $v$ on pseudoline $l$, the shortest $u,v$-path completely lies on $l$, and this path is unique.
\end{proposition}
	
Next, we study the shortest paths between any two vertices in a pseudoline arrangement graph. 
The following is a consequence of Proposition~\ref{prop:shortest}. 
\begin{proposition}\label{prop:number of lines in shortest path}
	For any two vertices $u$ and $v$, a shortest $u,v$-path of length $k$ has vertices on $k+2$ pseudolines.         
\end{proposition}
\begin{proof}
Let $P=uv_1v_2\ldots v_{k-1}v$ be a shortest $u,v$-path of length $k$. 
Traverse the path $P$ from $u$ to $v$ and count the new pseudolines encountered. 
At vertex $u$ we encounter two pseudolines. 
At each $v_i$, for $i \in [k-1]$, and at $v$ we encounter a new pseudoline, else by Proposition~\ref{prop:shortest}, the minimality of length of $P$ is contradicted. 
So we encounter $k+2$ pseudolines in total. 
\end{proof}    
	
Our next proposition is the analog of Proposition~\ref{prop:shortest}, for vertices that do not lie on a pseudoline.

\begin{proposition}\label{prop:shortestquadrant}
	For vertices $u$ and $v$ not on the same pseudoline, the shortest path between them completely lies in $Q_{uv}$. 
\end{proposition}
\begin{proof}
First, we claim that any shortest $u,v$-path lies in $Q_v$. 
Let $u=l_1\cap l_2$. 
For the sake of contradiction, let $P$ be a shortest $u,v$-path that does not lie completely in $Q_{v}$, that is, at least one vertex in $P$ lies \textit{outside} $Q_{v}$. 
By outside, we mean not even in $l_1$ or $l_2$. 
Thus there exists a vertex $v_f$ in $P$ that lies in $l_1$ or $l_2$ such that the vertex just before $v_f$ in $P$ lies outside $Q_{v}$. 
By Proposition~\ref{prop:shortest}, there is a strictly shorter $u,v$-path than $P$. 
This contradicts the minimality of $P$. 
		
Similarly, any shortest $u,v$-path lies in $Q_u$. 
Hence the shortest $u,v$-path lies in $Q_{uv}$.  
\end{proof}

Our next result shows that one of the eccentric vertices of any vertex lies on the outer face. 
	
\begin{proposition}\label{prop:eccentric}
	For a vertex $u \in V(G_{L})$, there exists an eccentric vertex of $u$ that is a $2$--vertex or $3$--vertex.
\end{proposition}
\begin{proof}
Let $v$ be an eccentric vertex of $u$, and let $d(u,v)=d$. 
Suppose $v$ is a $4$--vertex such that $v= l_3\cap l_4$. 
Then $u$ cannot be on $l_3$ or $l_4$, else one of the end vertices of $l_3$ or $l_4$ is farther from $u$ than $v$; a contradiction. 
		
Pseudolines $l_3$ and $l_4$ divide the plane into four pseudoquadrants, which we denote as $Q_u$, $Q_{uc}$, $Q'_u$ and $Q'_{uc}$ taken in a clockwise sense, such that $Q_u$ contains $u$ (see Figure \ref{fig19}). 
Let $x_{l_3}$ and $x_{l_4}$ be the neighbors of vertex $v$ on $l_3$ and $l_4$, respectively, in $Q'_u$, that is, $x_{l_3}\in Q'_u \cap Q_{uc}$ and $x_{l_4}\in Q'_u \cap Q'_{uc}$. 
So $l_3$ separates vertices $u$ and $x_{l_4}$, and $l_4$ separates vertices $u$ and $x_{l_3}$. 
		
\smallskip		
We prove the following observations. 
\begin{observation}\label{obs:1ecc}
	$d(u,x_{l_4})=d$ and $d(u,x_{l_3})=d$.
\end{observation}
\begin{proof}[of Observation \ref{obs:1ecc}]
We shall only prove the first equality; similarly, we can prove the second equality. 
Since $e(u)=d$, it follows that $d(u,x_{l_4})\leq d$. 
If $d(u,x_{l_4})<d$, then $d(u,x_{l_4})=d-1$.
Indeed, else if $d(u,x_{l_4}) < d-1$, then $d(u,v)_{\text{via } x_{l_4}}<d$; a contradiction. 
			
Let $P$ be a shortest $u,x_{l_4}$ path. 
Choose vertex $y\in l_3\cap P$ that is nearest to $u$, that is, $d(u,y)$ is minimum. 
The path $P$ does not contain $v$, else $d(u,x_{l_4}) > d$; a contradiction. 
			
Proposition \ref{prop:shortest} implies that $d(y,v)_{l_3} < d(y,v)_{\text{via $x_{l_4}$}}$, that is, $d(y,v)_{l_3} \leq d(y,x_{l_4})$. 
This implies $d(u,v)_{\text{via $y$}} \leq d-1$; and thus contradicts $d(u,v)=d$. 
Hence $d(u,x_{l_4})=d$. 
\end{proof}    
		
\begin{figure}
	\centering
	
	\begin{tikzpicture}[scale=0.7]
		\draw [solid] plot [smooth] coordinates{(0,0) (1,0.4)  (3,0) (5,-0.4) (7,0)};
		\draw [solid] plot [smooth] coordinates{(1,2) (2,1.2)  (3,0) (4,-1.2) (5,-2)};
		\draw [dashed] (2,-0.6) to[out=0,in=180+20] (2.7,-0.6) to [out=30,in=180+70] (3.2,-0.2) to[out=70,in=180+70] (3.3,0);

		\draw[fill] (3,0) circle (0.8pt);
		\draw[fill] (3.3,-0.07) circle (0.8pt);
		\draw[fill] (2.8,0.25) circle (0.8pt);
		\draw[fill] (2,-0.6) circle (0.8pt);
		\draw[fill] (3.2,-0.23) circle (0.8pt);
			
		\node at (7.2,0) {\footnotesize $l_4$};
		\node at (5.2,-2) {\footnotesize $l_3$};
		\node at (2.9,-0.2) {\footnotesize $v$};
		\node at (3.55,0.12) {\footnotesize $x_{l_4}$};
		\node at (2.92,0.6) {\footnotesize $x_{l_3}$};
		\node at (3.22,-0.5) {\footnotesize $y$};
		\node at (1.9,-0.7) {\footnotesize $u$};
		\node at (2,-1.5) {\footnotesize $Q_u$};
		\node at (4,1) {\footnotesize $Q'_u$};
		\node at (1,1.2) {\footnotesize $Q_{uc}$};
		\node at (5,-1.2) {\footnotesize $Q'_{uc}$};
	\end{tikzpicture}              
	\caption{Illustration for Proposition~\ref{prop:eccentric}}
	\label{fig19}
			
\end{figure}
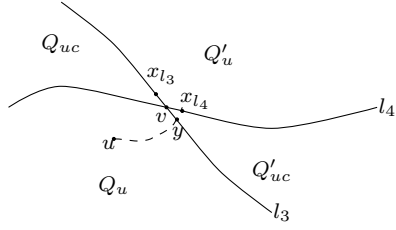

Before proceeding further, we modify the pseudoline arrangement to an \emph{isomorphic} one, where each pseudoline is a polyline with line segments between adjacent intersection points, and for every intersection point, each of the four angles between the two intersecting pseudolines (polylines)\footnote{One of the angles may be greater than $\pi$ if the point of intersection of the two pseudolines (polylines) is a point of non-differentiability for both of them.} is at most $\pi$. 
Such an arrangement can be obtained by considering the wiring diagram of the psuedoline arrangement and replacing the wires by line segments joining adjacent vertices. 
To see this, treat an intersection point in the wiring diagram as the origin of the co-ordinate axis system. 
Then each of its four quadrants has either an adjacent intersection point or the starting/ending point of one of the two pseudolines (if the origin is an end vertex). 
		
In such an arrangement we have the following observation.     
Let $\|p_1p_2\|$ denote the Euclidean distance between points $p_1$ and $p_2$. 
		
\begin{observation}\label{obs:2ecc}
	At least one of $\|ux_{l_3}\|$ or $\|ux_{l_4}\|$ is greater than $\|uv\|$. 
\end{observation}
\begin{proof}[of Observation \ref{obs:2ecc}]
If both $\|ux_{l_3}\|\leq \|uv\|$ and $\|ux_{l_4}\|\leq \|uv\|$, then $\angle x_{l_3}vx_{l_4}>\pi$. 
This contradicts our chosen arrangement. 
\end{proof}
		
Observations~\ref{obs:1ecc} and \ref{obs:2ecc} imply the existence of a neighbor of $v$, say $x(=x_{l_3}$ or $x_{l_4})$, which is also an eccentric vertex of $u$ such that $\|ux\|>\|uv\|$. 
If $x$ is a $4$--vertex, then we rename $x$ as the new $v$ and update the quadrants. 
		
By repeating the above arguments, we get a sequence of distinct vertices $x=x_1,x_2,\ldots$, such that $d(u,x_i)=d$ and $\|ux_i\| > \|ux_j\|$ for $i>j$. 
Since $|V(G_{L})|$ is finite, the last vertex of the sequence is not a $4$--vertex. 
Thus there exists an eccentric vertex of $u$ that is a $2$--vertex or $3$--vertex. 
\end{proof}

We immediately have the following corollary. 
\begin{corollary}
	For a vertex $u \in V(G_{L})$, there exists an eccentric vertex of $u$ that lies in the outer face of $R(G_{L})$. 
\end{corollary}
	
Next, we present the main results of this section. 
	
\subsection{Diameter and Characterization of Vertices with Maximum Eccentricity}\label{subsec:diameter and radius}
Finding the diameter of a pseudoline arrangement graph is a straightforward implication of Proposition~\ref{prop:number of lines in shortest path}. 
However, one can also prove it without using Proposition~\ref{prop:number of lines in shortest path}. 
For the sake of the reader, we restate Proposition~\ref{th:diameter}. 
	
\medskip

\noindent\textbf{Proposition 1.} \textit{The diameter of a pseudoline arrangement graph on $n$ lines is $n-2$.}

\begin{proof}[of Proposition \ref{th:diameter}]
\sloppy For vertices $u$ and $v$, Proposition~\ref{prop:number of lines in shortest path} implies that ${d(u,v)\leq n-2}$. 
The equality holds if $u$ and $v$ are the end vertices on a pseudoline. 
Hence diameter of a pseudoline arrangement graph realized on $n$ pseudolines is $n-2$. 
\end{proof}
	
In the above context, Proposition~\ref{prop:number of lines in shortest path} implies the following remark.
\begin{remark}\label{re:n-2path}
	If $d(u,v)= n-2-i$, then every shortest $uv$-path has vertices on $n-i$ pseudolines. 
	In particular, if $d(u,v)=n-2$, then any shortest $uv$-path has vertices on all the $n$ pseudolines. 
\end{remark}
	
For vertices $u$ and $v$, any shortest $u,v$-path has a vertex on every pseudoline that separates $u$ and $v$. 
Therefore, the number of such separating pseudolines lower bounds $d(u,v)$. 

\begin{remark}\label{re:sep}
	For vertices $u$ and $v$ in a pseudoline arrangement graph,  
		\[d(u,v) \geq 
		\begin{cases}
		\# \text{separating pseudolines} + 1, & \text{if $u$ and $v$ lie on the same pseudoline;}\\
		\# \text{separating pseudolines} + 2, & \text{if $u$ and $v$ do not lie on the same pseudoline.}    
		\end{cases}
		\]
\end{remark}
	
Theorem~\ref{th:outerlayer} characterizes the diametrical vertices of a pseudoline arrangement graph, that is, vertices with eccentricity equal to the graph diameter. 
For the sake of the reader, we restate Theorem~\ref{th:outerlayer}. 
	
\medskip
\noindent\textbf{Theorem 2.}
\textit{A vertex $v$ in a pseudoline arrangement graph $G$ is a diametrical vertex if and only if $v$ lies in the outer face of its realization $R(G)$. }
    
\subsubsection{Proof of Theorem~\ref{th:outerlayer}}\label{proof:outerlayer}
For the proof of sufficiency, if $u$ is a vertex in the outermost layer of a pseudoline arrangement realization, then it belongs to one of the unbounded regions $f$ in the arrangement. 
Let $v$ be a vertex of the unbounded region $f^*$ such that $f^*$ is separated from $f$ by every pseudoline. 
By Remark~\ref{re:sep} and Proposition~\ref{th:diameter}, $d(u,v)=n-2$. 
Hence $u$ is a diametrical vertex. 
	
\smallskip
For the proof of necessity, it suffices to show that if the vertex $u$ does not lie in the outermost layer, then its eccentricity is strictly less than $n-2$. 
Since $u$ is not on the outermost layer, there is a triplet of pseudolines $(l_1, l_2,l_3)$ such that $u$ lies in the pseudotriangle $T$ formed by them. 
Now consider any vertex $v$ in the arrangement. 
	
Depending on where $v$ lies with respect to $T$, we have the following four cases: (1) $v$ lies in the unbounded region of the induced arrangement on pseudolines $l_1, l_2,l_3$, that has all three of them in its boundary; (2) $v$ lies in the unbounded region of the induced arrangement on pseudolines $l_1, l_2,l_3$, that has exactly two of them in its boundary; (3) $v$ lies inside $T$; and (4) $v$ lies on the pseudolines $l_1, l_2,l_3$. 
(Using Proposition~\ref{prop:eccentric}, we can avoid case (3), however, it does not change the complexity of the proof.) 
We have to show that $d(u,v)<n-2$. 
	
	
\begin{figure}
	\centering
	\begin{minipage}{0.58\textwidth}
		\centering
		\begin{tikzpicture}[scale=0.75]
		
		\draw [dashed] (0,0) -- (0,4.5);
		\draw [dashed] (-1,0.5) -- (3,2.48);
		\draw [solid] (0,1) to[out=30,in=180+70] (0.7,2);
		\draw [solid] (0.7,2) to[out=70,in=180+70] (0.5,3.15);
		\draw [solid] (0.6,3.4) to[out=60,in=180+40] (2.5,3.16);
		\draw [solid] (2.7,3.38) to[out=70,in=180+70] (3.25,5.5);
		
		\draw [solid] plot [smooth] coordinates {(0,5) (1.8,3.8) (4,3.3)}; 
			
		\draw [dashed] (-0.5,4) -- (2.5,5.2) -- (4,5.8);
		\draw [dashed] (2.8,1) -- (3.3,6);
			
		\draw [solid] plot [smooth] coordinates{(-1.3,-1) (-0.35,0.7) (0.5,3.15) (0.6,3.4) (0.88,4.28) (1.4,5) }; 
		\draw [solid] plot [smooth] coordinates {(-1.4,-0.8) (0.55,1) (2.5,3.16) (2.7,3.38) (3.3,4) (4,4.4)}; 
			
		\draw[fill] (0,1) circle (1pt);
		\draw[fill] (3.25,5.5) circle (1pt);
		\draw[fill] (0.5,3.15) circle (1pt);
		\draw[fill] (0.6,3.4) circle (1pt);
		\draw[fill] (2.5,3.16) circle (1pt);
		\draw[fill] (2.7,3.38) circle (1pt);

		\node at (-0.2,4.4) {\footnotesize $l$};
		\node at (2.4,2) {\footnotesize $l'$};
		\node at (-0.2,5.2) {\footnotesize $l_3$};
		\node at (-1.1,-1) {\footnotesize $l_2$};
		\node at (-1.4,-0.5) {\footnotesize $l_1$};
		\node at (3.4,5.7) {\footnotesize $v$};
		\node at (-0.2,0.8) {\footnotesize $u$};
		\node at (0.35,3.25) {\footnotesize $p_0$};
		\node at (0.5,3.55) {\footnotesize $p$};
		\node at (2.6,3) {\footnotesize $q_0$};
		\node at (2.85,3.3) {\footnotesize $q$};
		\end{tikzpicture}
	    \caption{$v$ lies in the unbounded 3-face}
	    \label{fig:3face}
    \end{minipage}
    \begin{minipage}{0.4\textwidth}
		\centering
		\begin{tikzpicture}[scale=0.75]
			
		\draw [dashed] (0,0) -- (0,4.5);
		\draw [dashed] (-1,0.5) -- (3.8,2.8);
		\draw [solid] (0,1) to[out=30,in=180+70] (0.7,2);
		\draw [solid] (0.7,2) to[out=70,in=180+70] (0.5,3.15);
		\draw [solid] (0.6,3.4) to[out=60,in=180+40] (2.5,3.16);
		\draw [solid] (2.5,3.16)  to[out=20,in=180+40]  (4,3.7);
			
		\draw [solid] plot [smooth] coordinates {(0,5) (1.8,3.8) (4,3.3)}; 
			
		\draw [dashed] (-0.5,4) -- (2.5,5.2) -- (4,5.8);
		\draw [dashed] (2.8,1) -- (3.3,6);
			
		\draw [solid] plot [smooth] coordinates{(-1.3,-1) (-0.35,0.7) (0.5,3.15) (0.6,3.4) (0.88,4.28) (1.4,5) }; 
		\draw [solid] plot [smooth] coordinates {(-1.4,-0.8) (0.55,1) (2.5,3.16) (2.7,3.38) (3.3,4) (4,4.4)}; 
			
		\draw[fill] (4,3.7) circle (1pt);
		\draw[fill] (0,1) circle (1pt);
		\draw[fill] (0.5,3.15) circle (1pt);
		\draw[fill] (0.6,3.4) circle (1pt);
		\draw[fill] (2.5,3.16) circle (1pt);
		\draw[fill] (2.7,6) circle (0pt);

		\node at (-0.2,4.4) {\footnotesize $l$};
		\node at (2.4,2) {\footnotesize $l'$};
		\node at (-0.2,5.2) {\footnotesize $l_3$};
		\node at (-1.1,-1) {\footnotesize $l_2$};
		\node at (-1.4,-0.5) {\footnotesize $l_1$};
		\node at (4,3.5) {\footnotesize $v$};
		\node at (-0.2,0.8) {\footnotesize $u$};
		\node at (0.35,3.25) {\footnotesize $p_0$};
		\node at (0.5,3.55) {\footnotesize $p$};
		\node at (2.6,3) {\footnotesize $q_0$};
		\end{tikzpicture}
		\caption{$v$ lies in the unbounded 2-face}
		\label{fig:2face}
	\end{minipage}
\end{figure}

\smallskip
For the sake of contradiction, suppose $d(u,v)=n-2$. 
From Proposition~\ref{prop:shortestquadrant}, any shortest $u,v$-path is contained completely in the pseudo-4-gon $Q_{uv}$. 
Let $P$ be such a shortest $u,v$-path. 
From Proposition~\ref{prop:number of lines in shortest path}, for any pseudoline $l$, we have $l\cap Q_{uv}\neq \emptyset$ and $l$ contains at least a vertex of $P$. 
This adds more restrictions on the possible configurations for path $P$. 
Two possible scenarios (for cases (1) and (2)) are depicted in Figures~\ref{fig:3face} and \ref{fig:2face}, respectively. 
Note that, in Figure~\ref{fig:3face}, $p_0\neq p$ and $q_0\neq q$, and in Figure~\ref{fig:2face}, $p_0\neq p$, else three pseudolines meet at a point. 
Moreover the nature of path $q_0v|_P$ might vary, depending on the location of $v$ (the four cases); but it does not affect our proof.  
In all possible scenarios that are not contradicted by Proposition~\ref{prop:shortest}, the common structure is the path $uq_0|_P$ (see Figure~\ref{fig:common}). 
We have not highlighted the last two cases, where also, respecting Proposition~\ref{prop:shortestquadrant} and \ref{prop:number of lines in shortest path}, $uq_0|_P$ is the common structure. 
Thus it suffices to restrict our attention to this common path.

\begin{figure}[h]
	\centering
	\begin{minipage}{0.58\textwidth}
		\centering
		\begin{tikzpicture}[scale=0.75]
		
		\draw [dashed] (0,0) -- (0,4.5);
		\draw [dashed] (-1,0.5) -- (3,2.48);
		\draw [solid] (0,1) to[out=30,in=180+70] (0.7,2);
		\draw [solid] (0.7,2) to[out=70,in=180+70] (0.5,3.15);
		\draw [solid] (0.6,3.4) to[out=60,in=180+40] (2.5,3.16);
			
			
			
		\draw [solid] plot [smooth] coordinates{(-1.3,-1) (-0.35,0.7) (0.5,3.15) (0.6,3.4) (0.88,4.28) (1.4,5) }; 
		\draw [solid] plot [smooth] coordinates {(-1.4,-0.8) (0.55,1) (2.5,3.16) (2.7,3.38) (3.3,4) (4,4.4)}; 
			
		\draw[fill] (0,1) circle (1pt);
		\draw[fill] (0.5,3.15) circle (1pt);
		\draw[fill] (0.6,3.4) circle (1pt);
		\draw[fill] (2.5,3.16) circle (1pt);

		\node at (-0.2,4.4) {\footnotesize $l$};
		\node at (2.4,2) {\footnotesize $l'$};
		\node at (-1.1,-1) {\footnotesize $l_2$};
		\node at (-1.4,-0.5) {\footnotesize $l_1$};
		\node at (-0.2,0.8) {\footnotesize $u$};
		\node at (0.35,3.25) {\footnotesize $p_0$};
		\node at (0.5,3.55) {\footnotesize $p$};
		\node at (2.6,3) {\footnotesize $q_0$};
	\end{tikzpicture}
	\caption{Common structure in all the cases.}
	\label{fig:common}
	\end{minipage}
	\begin{minipage}{0.4\textwidth}
		\centering
		\begin{tikzpicture}[scale=0.75]
			
		\draw [dashed] (0,0) -- (0,4.5);
		\draw [dashed] (-1,0.5) -- (3,2.48);
		\draw [solid] (0,1) to[out=30,in=180+70] (0.7,2);
		\draw [solid] (0.7,2) to[out=70,in=180+70] (0.5,3.15);
		\draw [solid] (0.6,3.4) to[out=60,in=180+40] (2.5,3.16);
			
			
			
		\draw [solid] plot [smooth] coordinates{(-1.3,-1) (-0.35,0.7) (0.5,3.15) (0.6,3.4) (0.88,4.28) (1.4,5) }; 
		\draw [solid] plot [smooth] coordinates {(-1.4,-0.8) (0.55,1) (2.5,3.16) (2.7,3.38) (3.3,4) (4,4.4)}; 
			
		\draw[fill] (0,1) circle (1pt);
		\draw[fill] (0.5,3.15) circle (1pt);
		\draw[fill] (0.6,3.4) circle (1pt);
		\draw[fill] (2.5,3.16) circle (1pt);
		\draw[fill] (0.15,1.08) circle (1pt);
		\draw[fill] (1.05,1.53) circle (1pt);        
			
		\node at (-0.2,4.4) {\footnotesize $l$};
		\node at (2.4,2) {\footnotesize $l'$};
		\node at (-1.1,-1) {\footnotesize $l_2$};
		\node at (-1.4,-0.5) {\footnotesize $l_1$};
		\node at (-0.2,0.8) {\footnotesize $u$};
		\node at (0.35,3.25) {\footnotesize $p_0$};
		\node at (0.5,3.55) {\footnotesize $p$};
		\node at (2.6,3) {\footnotesize $q_0$};
		\node at (1.05,1.3) {\footnotesize $b$};
		\node at (0.17,0.91) {\footnotesize $f$};
	\end{tikzpicture}
	\caption{The closed polygon $fq_0|_{P_2}\cup fq_0|_{P_3}$}
	\label{fig:polygon}
	\end{minipage}
\end{figure}

\smallskip
Next we compare between two $u,q_0$-paths: the first path $P_1=P|_{uq_0}$ and the second path $P_2$ lies on pseudolines $l'$ and $l_1$. 
Let $b=l' \cap l_1$.    
Choose the vertex $f \in P_1 \cap P_2$ to be the farthest vertex from $u$ (that is, $d(u,f)$ is maximum) such that the $u,f$-path lies in $P_1\cap P_2$ (see Figure~\ref{fig:polygon}). Note that $f$ may be $u$. 
	
Vertex $f$ lies on $l'$. Indeed, if $f$ lies on $l_1$, then by Proposition~\ref{prop:shortest}, $P$ cannot be a shortest $u,v$-path; a contradiction. 
At this stage, we update the pseudolines $l_1$ and $l_2$ such that both $l\cap l_2$ and $l'\cap l_1$ are nearest to $u$, that is, $d(u,l\cap l_2)$ and $d(u,l'\cap l_1)$ is minimum. 
Thus none of the pseudolines intersect both $fb|_{l'}$ and $bq_0|_{l_1}$, else the choice of $l_1$ is contradicted.

\smallskip
Consider the closed pseudopolygon $fq_0|_{P_1}\cup fq_0|_{P_2}$, where $fq_0|_{P_2} = fb|_ {l'} \cup bq_0|_{l_1}$.  
Let the vertices on $fq_0|_{P_2}$ be the following $fx_0\ldots x_sbx_{s+1}\ldots x_tq_0$ such that $x_0\ldots x_s$ occur consecutively on $l'$ and $x_{s+1}\ldots x_t$ occur consecutively on $l_1$. 
Consider vertices $x_i$, for $0\leq i \leq t$, in an increasing order. 
For each $x_i$, with $0\leq i \leq s$, let $l_{x_i}$ denote the pseudoline such that $x_i = l_{x_i}\cap l'$; and for each $x_i$, with $s+1 \leq i \leq t$, let $l_{x_i}$ denote the pseudoline such that $x_i = l_{x_i}\cap l_1$.

The pseudoline $l_{x_i}$ also intersects $fq_0|_{P_1}$. 
Indeed, observe that $fq_0|_{P_1}\cup fq_0|_{P_2}$ is a closed pseudopolygon and there are no pseudolines that intersect both $fb|_{l'}$ and $bq_0|_{l_1}$ (because of our choice of $l_1$). 
And none of the pseudolines intersect $fb|_{l'}$ (respectively, $bq_0|_{l_1}$) twice as $l'$ (respectively, $l_1$) is a pseudoline. Hence every line intersecting $fq_0|_{P_2}$ also intersects $fq_0|_{P_1}$. 
	
Define the \textit{corresponding vertex} of $x_i$ in $fq_0|_{P_1}$ to be the vertex on $l_{x_i} \cap fq_0|_{P_1}$ that is closest to $x_i$. 
We claim that a vertex, say $t$, in $fq_0|_{P_1}$ can be the corresponding vertex of at most one $x_i$. 
Suppose a vertex $t$ in $fq_0|_{P_1}$ is the corresponding vertex of both $x_m$ and $x_n$. Then consider the vertex, say $s$, preceding $t$ in $fq_0|_{P_1}$, that is, $st$ is an edge in $fq_0|_{P_1}$.  
As $t$ is the corresponding vertex of both $x_m$ and $x_n$, $s$ does not lie on $l_{x_m}$ and $l_{x_n}$. 
Thus three pseudolines intersect at $t$, contradicting our assumption of a simple pseudoline arrangement. 
	
Next, we have the following observation. 
	
\begin{observation}\label{obs:corresponding vertex}
	Vertex $p_0$ is not a corresponding vertex of some $x_i$ in $fq_0|_{P_2}$. 
\end{observation}
\begin{proof}[of Observation \ref{obs:corresponding vertex}]
Suppose $p_0$ is a corresponding vertex of some $x_k$ in $fq_0|_{P_2}$. 
Then the edge of $fq_0|_{P_1}$, other that $p_0p$, that ends at $p_0$ cannot be on $l_k$, else by our definition the corresponding vertex of $x_k$ comes before $p_0$. 
It leads to a contradiction as three pseudolines intersect at $p_0$. 
\end{proof}
	
Observation~\ref{obs:corresponding vertex} implies that $d(u,q_0)|_{P_2} \leq d(u,q_0)|_{P_1}$. 
Thus $P_2 \cup q_0v|_{P}$ is a shortest $u,v$-path. 
But the pseudoline $l_2$ does not intersect it; so, by Remark~\ref{re:n-2path}, $d(u,v)<n-2$. 
This completes the proof of the necessity of Theorem~\ref{th:outerlayer} and hence the proof of Theorem~\ref{th:outerlayer}. \hfill $\square$
	
\medskip
As a direct consequence of Theorem~\ref{th:outerlayer}, we can also find the eccentricities of some vertices in the next layer.
Let $G_{L}$ be a pseudoline arrangement graph having vertices $V_{out}\subset V(G_L)$ on the outer face of its realization $R(G_L)$. 
The \emph{$1$--layer} of $R(G_{L})$ is the outer face of the realization upon removing all vertices in $V_{out}$ and their incident edges. 

Theorem \ref{th:outerlayer} implies that any interior vertex has eccentricity less than $n-2$. 
Notice that it is possible for vertices in the $1$--layer of $R(G_{L})$ to have no neighbors on the outer face. 
However, for vertices in the $1$--layer which have a neighbor on the outer face, we have the following corollary. 
	
\begin{corollary}
	Let $u_1$ be a vertex in the $1$--layer of $R(G_{L})$. If $u_1$ has a neighbor $u$ in the outer face of $R(G_{L})$, then $e(u_1)=n-3$.
\end{corollary}
\begin{proof}
Let $v$ be an eccentric vertex of $u$, that is, $d(u,v) = n-2$. 
If $e(u_1)<n-3$, then $d(u_1, v) < n-3$. 
This implies that $d(u,v) \leq d(u,u_1) +d(u_1,v) < n-2$; a contradiction. 
So $d(u_1,v) = n-3$. 
Since $u_1$ is an internal vertex, we have $e(u_1)=n-3$. \hfill 
\end{proof}    

\section{Final Remarks}\label{sec:final}
\subsection{Open Questions}\label{subsec:radius}
\paragraph{Degree Sequence.} 
In a general framework of questions involving degree sequences, a few more questions can be asked. 
Find (asymptotics of) the number of (pseudo) line arrangement graphs, which can be constructed that satisfy the conditions of Theorem~\ref{th:degreeSequence}?
As mentioned earlier, the separating examples for line arrangement graphs and pseudoline arrangement graphs imply that the answers are going to be different for the two graph classes. 
Moreover, can we construct some proportion of such graphs.  
	
\paragraph{The Question of Radius.}
As mentioned earlier, the major purpose of proving Theorem~\ref{th:outerlayer} was to find bounds on the radius of pseudoline arrangement graphs. 
The diameter of a pseudoline arrangement graph does not depend on the graph but the number of pseudolines in its realization. 
But one can be easily see that it would not be the case of the radius. 
A line arrangement graph with a centrally symmetric realization on $n$ lines will have a smaller radius than a line arrangement graph with a skewed realization on $n$ lines. 
We suspect it to follow:
$\big\lceil \frac{n}{2} \big\rceil -1 \leq r(G) \leq \big\lfloor \frac{3(n-1)}{4}\big\rfloor$ (for odd $n$). 
The upper bound comes from the star construction. 
Another open problem is to characterize the central vertices in a pseudoline arrangement graph, that is, vertices whose eccentricity equals the graph radius.  

\subsection{Alternate Proof of $d_2\geq 3$ in Theorem~\ref{th:degreeSequence} using Wiring Diagrams}\label{alternate d2>3 pseudoline}
Observe that the leftmost and rightmost intersection point in the wiring diagram of a pseudoline arrangement is always a $2$-vertex; so $d_2 \geq 2$. 
Next, we consider a `restricted wiring diagram' in which there is just one intersection point between the bottom two levels.  
In this case, observe that such an intersection point is also a $2$-vertex. 
This $2$-vertex differs from the leftmost and the rightmost intersection point in the wiring diagram (else one of the pseudolines has just one intersection point in it; a contradiction). 
Thus for a pseudoline arrangement with a wiring diagram that has only one intersection point between the bottom two levels, $d_2 \geq 3$. 
We claim that all pseudoline arrangements have such a restricted wiring diagram. 
To show this, we need to carefully set up the topological sweep that fixes the wiring diagram. Choose the sweep line to originate from an unbounded face that is bounded by two pseudolines in the pseudoline arrangement (this always exists as $d_2 \geq 2$). 
Perform the topological sweep to form the required restricted wiring diagram. 

\begin{figure}
	\centering
	
	\includegraphics[scale=0.7]{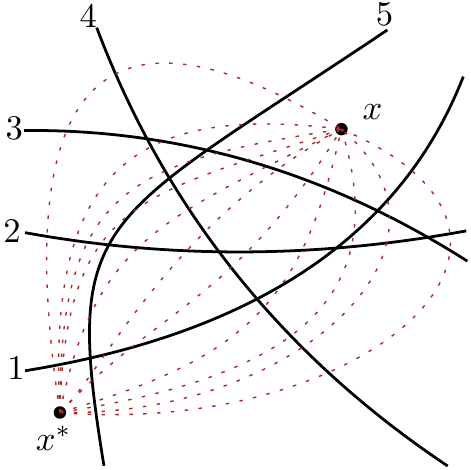}\hspace{1cm} \includegraphics[scale=0.7]{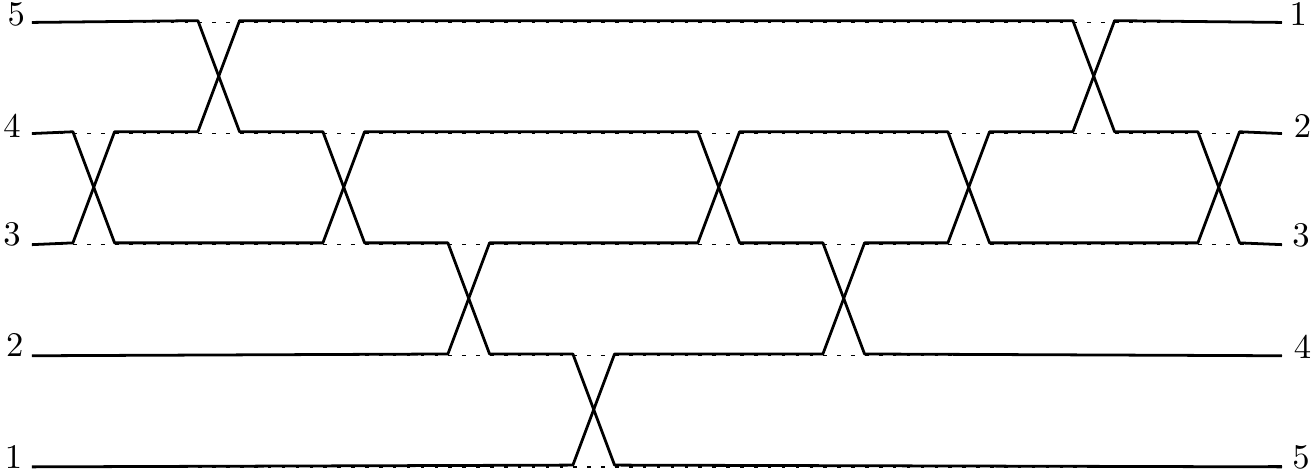}
	\caption{Restricted wiring diagram.}
	\label{fig:res}
	
\end{figure}

\subsection{A Brief Review on Pseudoline Arrangements} \label{ssec: survey pseudoline arrangements}
As mentioned earlier, two of the main reasons driving the study of pseudoline arrangements are (i) the numerous problems and conjectures collected in the survey book by Gr\"unbaum~\cite{grunbaum_arrangements_1972}, and (ii) a consequence of the topological representation theorem of Folkman and Lawrence~\cite{folkman_oriented_1978} that gives a geometric interpretation of oriented matroids of rank three in terms of pseudoline arrangements (see the standard text by Bj\"orner et al.~\cite[Chap.~6]{bjorner_oriented_1999} for detailed discussions). 
The question on the enumeration of non-isomorphic pseudoline arrangements by Knuth has produced a series of rich results by Knuth~\cite{knuth_axioms_1992}, Felsner~\cite{felsner_number_1997}, Felsner and Valtr~\cite{felsner_coding_2011}, and Dumitrescu and Mandal~\cite{dumitrescu_new_2020}. 
Another interesting combinatorial question asks to find the number of triangles in line arrangements. This has been studied by Melchior~\cite{melchior_uber_1941}, Levi~\cite{levi_teilung_1926}, F\"uredi and Palsti~\cite{furedi_arrangements_1984}, Felsner and Kriegel~\cite{felsner_triangles_1999}, and others. 

For further details on line and pseudoline arrangements, see the surveys by 
Erd\H{o}s and Purdy~\cite{erdos_extremal_1995}, and the latest survey by Felsner and Goodman~\cite{felsner_pseudoline_2004}; also see the book by Felsner~\cite[Chap. 5 and 6]{felsner_geometric_2004}. 
Gr\"unbaum~\cite{grunbaum_arrangements_1972} (in the 1970s) was the first to collect relevant results and posed many problems and conjectures on arrangements of lines as well as pseudolines. 
The chapter by Erd\H{o}s and Purdy~\cite{erdos_extremal_1995} also addresses various aspects of arrangements. 
The survey by Felsner and Goodman~\cite{felsner_pseudoline_2004} is the most recent (2017). 
Chapter~5 in Felsner's book~\cite{felsner_geometric_2004} contains results on line arrangements, whereas pseudoline arrangements are studied in Chapter~6. 
Algorithms on arrangements are addressed in the book~\cite{edelsbrunner_algorithms_1987} by Edelsbrunner (also see the book by Matou\v{s}ek~\cite[Chap.~6]{matousek_lectures_2002}). 


\smallskip
Moving on to graph-theoretic aspects, apart from the graph classes defined in this article, many researchers have studied other graph/hypergraph classes obtained from the underlying pseudoline arrangements. Tamaki and Tokuyama~\cite{tamaki_characterization_2003} defined a graph class \emph{pseudoline graph} with the pseudolines as the vertex set and characterized planar graphs as \textit{diamond}-free pseudoline graphs. Eppstein et al.~\cite{eppstein_convex-arc_2013} considered graphs obtained from embeddings of \textit{weak} simple pseudoline arrangements (two pseudolines intersect at most once). They showed that the realizations of so-obtained outerplanar graphs can be represented by a set of chords in a circle. Bose et al.~\cite{bose_coloring_2013} considered various hypergraphs obtained from simple line arrangements and studied their chromatic number and minimum vertex cover.

\subsection{A Brief Review on Degree Sequence-based Characterizations}\label{subsec:degreeSequence}
For a given property $P$, invariant under isomorphism, a degree sequence is said to be (1) \emph{potentially P-graphic} if at least one of its realizations satisfies property $P$ and (2) \emph{forcibly P-graphic} if all its realizations satisfy property $P$. 
Barrus, Kumbhat, and Hartke~\cite{barrus_graph_2008} characterized the graph classes $\mathcal{F}$ such that potentially $\mathcal{F}$-free sequences are also forcibly $\mathcal{F}$-free. 
In this survey, we focus on the recognition of graph classes based on their degree sequences, hence omitting those characterizations that explore the relationship between degree sequences and various graph properties. 
These degree sequence-based recognition characterizations are of two types depending on whether at least one or all the realizations of the degree sequence satisfy the given conditions. 
The latter is known as \emph{degree sequence characterization}, that is, it tells us whether a graph belongs to the graph class solely based on its degree sequence.  
The class of graphs that have a degree sequence characterization is closed under the 2-switch operation. 
Note that a degree sequence characterization is different from forcibly P-graphic characterization (which we shall not cover here): the former characterizes graph classes whereas the latter characterizes the degree sequences. 
Standard, but old, surveys are by Hakimi and Schimeichel~\cite{hakimi_graphs_1978} (in 1978) and Rao~\cite{rao_survey_1981} (in 1980).

The result on trees is folklore. 
In 1979, Beineke and Schmeichel~\cite{beineke_degrees_1971} characterized degree sequences of graphs with one cycle. 
It was generalized to cacti graphs by Rao~\cite{rao_degree_1981} in 1981. 
Much later in 2005, B{\'{\i}}y{\'{\i}}ko\u{g}lu~\cite{biyikoglu_degree_2005} further generalized this result to Halin graphs. 
Bose et al.~\cite{bose_characterization_2008} characterized degree sequences of 2-trees in 2008. 
	
Moving on to other graph classes, Hammer and Simeone~\cite{hammer_splittance_1981} in 1981, and Merris~\cite{merris_split_2003} in 2003 gave a degree sequence characterization of split graphs resulting in a linear time recognition algorithm. 
Here we note two things (1) if $G$ is a split graph, then every graph with the same degree sequence as $G$ is also a split graph; and (2) the degree sequence of a split graph does not determine the graph up to isomorphism. 
However, in the case of \emph{threshold graphs}, Hammer and others~\cite{chvatal_set-packing_1973, hammer_degree_1978}, in 1973 and 1978 respectively, gave a degree sequence characterization exploiting the fact that the structure of the threshold graph is completely described by its degree sequence.

\acknowledgements
\label{sec:acknowledgement}
The authors thank Prof. Douglas B. West for his encouragement to pursue the line arrangement graph realization problem, which was the starting point of this work. The authors also thank Dibyayan Chakraborty for suggesting to pursue the questions on eccentricity.  
The authors also thank the reviewers for greatly enhancing both the content and the presentation
of this paper. 
The first and third authors are partially supported by IFCAM project Applications of graph homomorphisms (MA/IFCAM/18/39).

\nocite{*}
\bibliographystyle{plainurl}
\bibliography{sample-dmtcs}
\label{sec:biblio}

\section{Appendix}
	\subsection{Proof of Proposition~\ref{prop:shortest}}\label{proof:prop:shortest}
	\begin{proof}
		We proceed by strong induction on $d(a,b)_l$. 
		For the base case, when $d(a,b)_l=1$, vertices $a$ and $b$ are adjacent on $l$. 
		So $P=ab|_l$ is the unique shortest $a,b$-path. 
		As our induction hypothesis, we assume that for $d(a,b)_l<k$, $P=ab|_l$ is the unique shortest $a,b$-path. 
		
		Now let $d(a,b)_l=k$. 
		Suppose there exists another shortest $a,b$-path $P'$ ($\neq P$). 
		We shall show that ${d(a,b)|_P < d(a,b)|_{P'}}$, thereby contradicting the existence of $P'$. 
		Proposition~\ref{prop:shortest} is implied from the following observations.

		\begin{observation}\label{obs:1line}
			If $V(P) \cap V(P') \supsetneq \{a,b\}$, then $P'=P$. 
		\end{observation}    
		\begin{proof}[of Observation \ref{obs:1line}]
			Suppose $V(P) \cap V(P')\supsetneq \{a,b\}$; then there exists $c \in V(P) \cap V(P')$ with $c\notin \{a,b\}$. 
			Let $P_1=ac|_P$, $P_2=cb|_P$, $P_1'=ac|_{P'}$ and $P_2'=cb|_{P'}$. 
			Since $d(a,c)_l < k$, by our induction hypothesis, $ac|_l$ is the unique shortest $a,c$-path. 
			Similarly, $cb|_l$ is the unique shortest $c,b$-path.
			Thus $P_1=P_1'$ and $P_2=P_2'$; hence $P=P'$. 
		\end{proof}
		
		\medskip    
		So we are left with the case where $V(P) \cap V(P')= \{a,b\}$.
		
		\begin{observation}\label{obs:2line}
			If $V(P) \cap V(P')= \{a,b\}$, then $|P'|>|P|$.
		\end{observation}
		\begin{proof}[of Observation \ref{obs:2line}]
			Let $P=ab|_l=aa_1a_2\ldots a_{k-1}b$ and $P'=ab_1b_2\ldots b_sb$ such that $V(P) \cap V(P')= \{a,b\}$, then it suffices to prove that $s>k-1$.

			Consider the path $P'$. 
			For $i \in [s]$, let $l_i$ represent the pseudoline at $b_i$ that does not contain the edge $b_{i-1}b_i$ for $i>1$, or the edge $ab_1$ for $i=1$. 
			
			First we claim that $l_i\neq l_j$, for $i\neq j$ and $i,j \in [s]$, that is, each pseudoline $l_i$ is unique. 
			Suppose for some $b_t$, with $t\leq s$, the pseudoline encountered, $l_t$, is not unique. 
			Then there exists some $b_r$, with $r\neq t$, such that $l_r=l_t$. 
			Our induction hypothesis implies that $b_rb_{r+1}\ldots b_t$ (without loss of generality assume $r<t$) lies on $l_t$. 
			In particular, $b_{t-1}b_t$ lies on $l_t$; a contradiction (by definition of $l_t$). 
			Hence at each $b_i$, we encounter an unique pseudoline $l_i$. 
			
			For each $j \in [k-1]$, the line intersecting $l$ at $a_j$ is some $l_i$, for $i\in [s]$ (that contains vertex $b_i$ in $P'$).  
			This occurs as $P\cup P'$ is a closed curve with $P$ being on a pseudoline, and hence no other pseudoline intersects $P$ twice. 
			Thus $s \geq k-1$. 
			For strict inequality observe that there exists a $b_u$ with $u\leq s$, such that $l_u\cap l =\{b\}$.
			This $l_u$ does not contain any $a_j$. 
			Thus $s>k-1$. 
		\end{proof}
		
		\medskip     
		Observations~\ref{obs:1line} and \ref{obs:2line} imply that $P$ is the unique shortest $a,b$-path. \hfill 
	\end{proof}    
\end{document}